\documentclass[11pt]{article}
\usepackage{amsmath,amssymb,amsthm,amsfonts,amscd}
\usepackage{fullpage}
\usepackage[numeric]{amsrefs}
\usepackage{rotating} % For the rotated tables ("sidewaystable")
\input colordvi
\usepackage{graphicx}
\newtheorem{thm}[equation]{Theorem}
\newtheorem{cor}[equation]{Corollary}

\numberwithin{equation}{section}

\renewcommand\a{\alpha}

\newcommand\g{\gamma}

\newcommand\e{\varepsilon}
\renewcommand\l{\lambda}

\newcommand\D{\Delta}

\renewcommand\D{\Delta}

\newcommand\f{\frac}

\newcommand{\Z}{{\mathbb{Z}}}
\newcommand{\R}{{\mathbb{R}}}

\newcommand{\C}{{\mathbb{C}}}
\newcommand{\A}{{\mathbb{A}}}
\newcommand{\Q}{{\mathbb{Q}}}

\renewcommand\Re{\text{Re~}}

\newcommand\srel[2]{\begin{smallmatrix} {#1} \\ {#2} \end{smallmatrix}}

\newcommand{\gobble}[1]{}
  \newcommand{\rangeref}[2]{%
    \ref{#1}--\afterassignment\gobble\fam 0\ref{#2}%
  }

\makeatletter
\def\imod#1{\allowbreak\mkern5mu({\operator@font mod}\,#1)}
\makeatother

\begin{document}

\title{Residual automorphic forms and spherical unitary representations of exceptional groups}

\date{April 29, 2012}

\author{Stephen D. Miller\thanks{Supported by NSF grant DMS-0901594.}\\
Rutgers University\\
\tt{miller@math.rutgers.edu}}

\maketitle

\begin{abstract}

Arthur has conjectured that the unitarity of a number of representations can be shown by finding appropriate automorphic realizations.
 %Vogan has further suggested they occur as specific residues of Eisenstein series.
  This has been verified
 for classical groups by M\oe glin  and for the exceptional Chevalley group $G_2$ by Kim.  In this paper we extend their results on spherical representations to the remaining exceptional groups  $E_6$, $E_7$, $E_8$, and $F_4$.  In particular we prove Arthur's conjecture that the spherical constituent of an unramified principal series of a Chevalley group over any local field of characteristic zero is  unitarizable if its Langlands parameter coincides with half the weighted marking of a coadjoint nilpotent orbit of the Langlands dual Lie algebra.

\vspace{.2cm}

\noindent keywords:~Arthur's conjectures, unitary dual, residual Eisenstein series, automorphic realizations, unipotent representations, small representations.

\end{abstract}

\section{Introduction}\label{sec:intro}

The most trivial automorphic form -- the constant function on the complex upper half plane -- has a complicated construction as the residue of the usual nonholomorphic Eisenstein series $\sum_{(c,d)\in\Z^2-\{(0,0)\}}\f{y^s}{|cz+d|^{2s}}$ at the polar point $s=1$.  The constant residue reflects the fact that the principal series representation associated to this Eisenstein series has a trivial quotient.  This is the starting point for a fascinating  mechanism of constructing automorphic realizations of certain ``small'' representations of real groups, meaning those with low Gelfand-Kirillov dimension (equivalently, those whose wavefront set is small).  For example, the classical Jacobi $\theta$-function is a residue of an Eisenstein series on the metaplectic double cover of $SL(2,\R)$.  When the residues are square-integrable (which can be checked by a criterion of Langlands reviewed in section~\ref{sec:Langlands}), they lie in the discrete automorphic spectrum and  hence automatically give unitary representations.  This strategy was famously used by Speh \cite{Speh} to construct new unitary representations of $SL(4,\R)$, representations which were difficult to approach until she added this arithmetic grasp.

As part of a broad set of conjectures,
Arthur \cite{Arthur} has proposed that  spherical constituents of principal series representations
at certain points of reduction are unitary; moreover, their unitarity should, as above, be a consequence of an automorphic realization. In the case of unramified representations of split groups,  one possibility is a precise realization of these representations as residues of Eisenstein series.   This reduces Arthur's conjectures to verifying certain residual Eisenstein series are square-integrable.  In this paper we prove the square-integrability and hence
these conjectures, whose statement we now recall.

Let $G$ denote  a Chevalley group and   $B=NA$ be a fixed minimal parabolic subgroup of $G$, with $N$ a maximal unipotent subgroup and $A$ a  maximal torus.  We shall assume, as we may, that
$N$ is generated by the one-parameter subgroups for each of the Chevalley basis positive root vectors, and that the Lie algebra $\frak a$ of $A$ is spanned by the Chevalley basis coroot vectors.
To any coadjoint nilpotent orbit $\cal O^\vee$ in the complexified Lie algebra  $\frak g^\vee \otimes \C$ of the Langlands dual group $G^\vee$, there is a  linear functional $2\l_0(\cal O^\vee)$ on  ${\frak a}(\C)$  coming from  its marked Dynkin diagram.
For example, in the case of  the  ``regular orbit'' (i.e., the unique dense orbit)
 $\l_0(\cal O^\vee)$ is equal to $\rho$, half the sum of the positive roots.  Suppose furthermore that $\cal O^\vee$ is {\em distinguished}, meaning  it does not intersect any proper Levi subalgebra. Let $F$ be a number field and $\A_F$ its ring of adeles.
 Arthur conjectured that for any place $v$ of $F$, the spherical constituent of the unramified principal series representation of $G(F_v)$ with Langlands parameter $\l_0(\cal O^\vee)$ is unitarizable, and moreover occurs discretely in  the automorphic spectrum $L^2(G(F)\backslash G(\A_F))$.  The situation when ${\cal O}^\vee$ is not distinguished reduces to this; see corollary~\ref{thecor} and the remarks following it.

 The discrete spectrum includes cusp forms, which are very difficult to construct, but can sometimes be counted using the trace formula (or more commonly, variants of the trace formula).  The rest of the discrete spectrum consists of residual Eisenstein series at special, delicate points (see \cite{MWbook} for a detailed general reference).  These are automorphic forms occurring as the leading coefficient in a multivariable Laurent series expansion of   Eisenstein series, which are automorphic realizations of principal series representations.  We now state the definition of the spherical, unramified minimal parabolic Eisenstein series, which are the relevant type in this paper.  The adjoint action of $A(\A_F)$ on the Chevalley basis simple root vectors, composed with the global valuation on $\A_F$, gives rise to a character $a\mapsto a^\lambda$ of $A(F)\backslash A(\A_F)$ for each   $\l\in{\frak a}^*\otimes\C$, the complex span of the simple roots.  Let $a(g)$ denote the Iwasawa $A$-factor of $g\in G(\A_F)$.
  The (unramified) minimal parabolic Eisenstein series is defined as
\begin{equation}\label{eisdef}
    E(\l,g) \ \ := \ \ \sum_{\g\,\in\,B(F)\backslash G(F)}a(\g g)^{\l+\rho}\ , \ \  \ \ \l \, \in \, {\frak a}^*\otimes \C\,,
\end{equation}
initially as an absolutely convergent sum when  $\l-\rho$  lies in the interior of the positive Weyl chamber, and then for general $\l$ by meromorphic continuation.  At values of $\l$ for which $E(\l,g)$ is holomorphic, it and its right translates generate an automorphic realization of each the local principal series representations $V_{\l}=\{f:G(F_v)\rightarrow \C\,|\, f(nag)=a^\l f(g)\,, n\in N(F_v)\,,a\in A(F_v)\}$ for each place $v$ of $F$.

   Arthur's conjectures suggest  that the residues of $E(\l,g)$ at $\l=\l_0(\cal O^\vee)$ for distinguished orbits $\cal O^\vee$ lie in $L^2(G(F)\backslash G(\A_F))$ (though they would not be contradicted were this false, as there would remain the possibility of cuspidal realizations).  Such a function then generates an irreducible subrepresentation of $L^2(G(F)\backslash G(\A_F))$, whose unitarity comes from the $L^2$ inner product.  Due to the agreement of infinitesimal characters, it must be the spherical constituent of each local principal series representation $V_{\l_0({\cal O}^\vee)}$ over all completions $F_v$ of $F$.
Our main result is a proof of this $L^2$ property for the exceptional groups $E_6$, $E_7$, $E_8$, and $F_4$ (it is known for  classical groups and $G_2$ \cite{jacquet,MoeglinWaldspurger,Moeglin,Kim}).
We summarize our findings as follows:

\begin{thm}\label{findings}
Let $F$ be a number field, $G$ be a Chevalley group of type $E_6$, $E_7$, $E_8$, or $F_4$, and $\frak g$ its Lie algebra.  Let $\l_0(\cal O^\vee)\in {\frak a}^*\otimes\C$ and a coadjoint nilpotent orbit $\cal O$ in $\frak g(\C)$ be as defined by one of the following pairs, in which $\omega_1,\omega_2,\ldots$ refer to  fundamental weights in the usual Bourbaki numbering and the orbit $\cal O$ is described by its Bala-Carter label (see \cite{Collingwood}):
\begin{center}\begin{tabular}{|c|c|c||c|c|c|}
\hline
 $G$ &  $\l_0(\cal O^\vee)$ &  $\cal O$ &  $G$ &  $\l_0(\cal O^\vee)$ & $\cal O$ \\
\hline
 $E_6$ &  $\omega_1+\omega_4+\omega_6$ &  $A_2$ &  $E_8$ &  $\omega_5$ & $E_8(a_7)$ \\
 $E_6$ &  $\rho-\omega_4$ &  $A_1$ &  $E_8$ &  $\omega_4+\omega_8$ & $D_4(a_1)+A_2$ \\
 $E_6$ &  $\rho$ &  $0$ &  $E_8$ &  $\omega_4+\omega_7$ & $D_4(a_1)+A_1$ \\
 $E_7$ &  $\omega_4+\omega_7$ &  $D_4(a_1)$ &  $E_8$ &  $\omega_4+\omega_7+\omega_8$ & $D_4(a_1)$ \\
 $E_7$ &  $\omega_1+\omega_4+\omega_7$ &  $A_2+2A_1$ &  $E_8$ &  $\omega_1+\omega_4+\omega_7$ & $2A_2$ \\
 $E_7$ &  $\omega_1+\omega_4+\omega_6+\omega_7$ &  $A_2$ &  $E_8$ &  $\omega_1+\omega_4+\omega_7+\omega_8$ & $A_2+2A_1$ \\
 $E_7$ &  $\rho-\omega_4-\omega_6$ &  $2A_1$ &  $E_8$ &  $\omega_1+\omega_4+\omega_6+\omega_8$ & $A_2+A_1$ \\
 $E_7$ &  $\rho-\omega_4$ &  $A_1$ &  $E_8$ &  $\rho-\omega_2-\omega_3-\omega_5$ & $A_2$ \\
 $E_7$ &  $\rho$ &  $0$ &  $E_8$ &  $\rho-\omega_4-\omega_6$ & $2A_1$ \\
 $F_4$ &  $\omega_3$ &  $F_4(a_3)$ &  $E_8$ &  $\rho-\omega_4$ & $A_1$ \\
 $F_4$ &  $\omega_1+\omega_3$ &  $A_1+A_1s$ &  $E_8$ &  $\rho$ & $0$ \\
 $F_4$ &  $\rho-\omega_2$ &  $A_1s$ &   &   &  \\
 $F_4$ &  $\rho$ &  $0$ &   &   &  \\
\hline
\end{tabular}\end{center}
Then the unramified Borel Eisenstein series (\ref{eisdef})  has a square-integrable residue
 at $\l=\l_0(\cal O^\vee)$.
 Its local representation of $G(F_v)$  is unitary for each place $v$ of $F$, and furthermore has  wavefront set ${\cal O}$ if $v$ is archimedean.
\end{thm}

\noindent The case of $\l_0(\cal O^\vee)=\rho$ (which has a trivial residue) is of course well-known.  The cases with ${\cal O}=A_1$ for $E_6$, $E_7$, and $E_8$ are the automorphic realizations of the minimal representation constructed in  \cite{GRS}, and the cases with ${\cal O}=2A_1$ for $E_7$ and $E_8$ are likewise the automorphic realizations of the ``next-to-minimal'' representation constructed  in \cite{GMV}.  Indeed, both the proof and immediate motivation for writing this paper arose out of the collaboration \cite{GMV}.  The appearance of small automorphic residual representations in certain string theory problems there and in \cite{GMRV} led us to develop computational tools used here.  The theorem's assertion about the wavefront set is a direct consequence of Theorem A.5 of \cite[Appendix A]{GMV}, by  Ciubotaru and  Trapa.

Theorem~\ref{findings} has well-known implications for the unitarity of spherical representations coming from nondistinguished orbits (which was shown by Barbasch-Moy \cite{BM1} for nonarchimedean fields):

\begin{cor}\label{thecor}
Let $2\l_0(\cal O^\vee)$ be the weighted marking of any coadjoint nilpotent orbit of ${\frak g}^\vee\otimes\C$, distinguished or not.  Then the spherical constituent of $V_{\l_0(\cal O^\vee)}$ over any local field of characteristic zero is unitarizable.
\end{cor}
\begin{proof}  Every local field is a completion $F_v$ of some number field $F$.
By theorem~\ref{findings} we need only consider the case when $\cal O^\vee$ intersects the complexified Lie algebra $\frak l(\C)$ of   the Levi component of a proper parabolic subgroup $P$, which we may assume to be minimal among such parabolics.  We furthermore may assume $\frak l$ is chosen  compatibly with the Chevalley basis.  Then ${\cal O}^\vee_L={\cal O^\vee}\cap {\frak l}(\C)$ is distinguished in $\frak l(\C)$, and so by Theorem~\ref{findings} and its known analog for classical groups it is associated to a spherical unitary representation of $L(F_v)$.  Unitary induction of this representation from $L(F_v)$ to $G(F_v)$ gives a unitary representation which contains the spherical constituent of $V_{\l_0(\cal O^\vee)}$.  Indeed, this last statement reduces to the compatibility of their respective infinitesimal characters; that, in turn,   is equivalent to the agreement of the weighted marking of the nondistinguished orbit $\cal O^\vee$ of $\frak g^\vee(\C)$ with the weighted marking of the distinguished orbit ${\cal O}^\vee_L$ of $\frak l(\C)$.
\end{proof}
The unitary induction in the proof has an automorphic analog, as Eisenstein series induced from automorphic forms on the Levi component of a proper parabolic subgroup.  Thus the representations in the corollary also occur automorphically.

\vspace{.3 cm}

{\bf Acknowledgements:} I would like to thank James Arthur, Dan Barbasch, Dan Ciubotaru, Joseph Hundley, David Kazhdan,Dragan Milicic, Steven J. Miller,  Colette M\oe glin, Wilfried Schmid, Freydoon Shahidi, Peter Trapa, and David Vogan for their helpful discussions and explanations.  I would also like to thank my collaborators Michael Green and Pierre Vanhove for their stimulating suggestions.

\section{Langlands' constant term formula and $L^2$ condition}\label{sec:Langlands}

Let $\Delta^{+}$ and $\D^{-}$ denote the  positive and negative roots, respectively, of the Chevalley group $G$ with respect to its fixed minimal parabolic $B$, and  $\Sigma^+$ its positive simple roots.  We use the notation $\a^\vee$ for the coroot  of $\a\in\D^+$.
Langlands computed the constant term of the (unramified) minimal parabolic Eisenstein series (\ref{eisdef}) as
\begin{equation}\label{constanttermformula}
    \int_{N(F)\backslash N(\A_F)}E(\l,ng)\,dn \ \  = \ \ \sum_{w\,\in\,W}M(w,\l)\,a(g)^{w\l+\rho}\,,
\end{equation}
where $W$ is the Weyl group,
\begin{equation}\label{Mwlambdadef}
    M(w,\l) \ \ = \ \ \prod_{\srel{\a\,\in\,\D^{+}}{w\a\,\in\,\D^{-}}}c(\langle \l,\a^\vee\rangle)\,,
\end{equation}
and $c(s)$ is a meromorphic function on $\C$ which can be expressed as an explicit ratio involving the Dedekind $\zeta$-function of the number field $F$ (see \cite[\S6]{shahidi} and \cite[\S 3.7]{bump}.)
We shall not require this formula, but only its following direct consequences:~$c(s)$  has a first order zero at $s=-1$, a first order pole at $s=1$,  no zeroes or poles in $\{\Re{s}<-1\}$, and  satisfies  $c(s)c(-s)=1$, $c(0)=-1$.
 Langlands also showed the corresponding functional equation
\begin{equation}\label{eisfe}
    E(\l,g) \ \ = \ \ M(w,\l)\,E(w\l,g)
\end{equation}
of the meromorphic function $\l\mapsto E(\l,g)$, $\l\in{\frak a}^*\otimes \C$.
%
%------------------------
%
%%
%%Let $\zeta_F$  denote the Dedekind $\zeta$-function of the number field $F$,  and let $\xi_F$ denote its completion (i.e., with $\Gamma$-factors).  Let
%%\begin{equation}\label{c-func-def}
%%    c(s) \ \ := \ \ \f{\xi_F(s)}{N^{1/2}\,\xi_F(s+1)} \ \ = \ \  \f{\xi_F(1-s)N^{s-1/2}}{\xi_F(-s)N^{s+1/2}}
%%\end{equation}
%%
%%
%%  Generalizing (\ref{eisdef}) the Borel Eisenstein series for $G$ is the automorphic form
%%\begin{equation}\label{eisdef2}
%%    E(\l,g) \ \ := \ \  \sum_{\g\,\in\,B(F)\backslash G(F)} e^{(\l+\rho)(H(\g g))}\,,
%%\end{equation}
%%defined in the Godement range of convergence $\{\langle \l-\rho,\a\rangle > 0,\a\in \Sigma^{+}\}$.
%
% Langlands showed that it has a meromorphic continuation to the complex plane, with functional equation
%\begin{equation}\label{eisfe}
%    E(\l,g) \ \ = \ \ M(w,\l)\,E(w\l,g)
%\end{equation}
%under the Weyl group $W$ of $G$.  Here $M(w,\l)$ are intertwining operators satisfying
%\begin{equation}\label{Mwlambdadef}
%    M(w,\l) \ \ = \ \ \prod_{\srel{\a\,\in\,\D^{+}}{w\a\,\in\,\D^{-}}}c(\langle \l,\a^\vee\rangle)\,.
%\end{equation}
%These also appear in the constant term formula
%\begin{equation}\label{constanttermformula}
%    \int_{N(F)\backslash N(\A)}E(\l,ng)\,dn \ \  = \ \ \sum_{w\,\in\,W}M(w,\l)\,e^{(w\l+\rho)(H(g))}\,.
%\end{equation}
%The scattering function is a ratio of two Dedekind $\zeta$-functions, and satisfies the property $c(s)c(-s)=1$.  It has a simple zero at $s=-1$, a simple pole at $s=1$, and no other zeros or poles outside of $0<\Re{s}<1$.

Any coefficient in a  (multivariable) Laurent expansion of an Eisenstein series in $\l$ is an automorphic function in $g$.  In particular suppose that $\l$ has the form $\l_1+\e\l_2$ for $\e\in\C$ and fixed $\l_1,\l_2\in{\frak a}^*\otimes \C$, and that $E(\l,g)$ has a zero of order $n$ at $\e=0$ (by convention, $n<0$ when there is a pole).   Grouping terms with similar powers together,
\begin{equation}\label{grouping1}
\aligned
     \int_{N(F)\backslash N(\A_F)}E(\l_1+\e\l_2,ng)\,dn \ \ & = \ \ \sum_{\mu\,\in\,W\l_1}a(g)^{\mu+\rho}\sum_{w\,\in\,W(\l_1,\mu)}
    M(w,\l_1+\e\l_2)\,a(g)^{\e w \l_2} \\
    & = \ \ \sum_{\mu\,\in\,W\l_1}a(g)^{\mu+\rho} \sum_{k \,\ge\,n} \e^k\,C(\mu,k,g)\,,
\endaligned
\end{equation}
where $W\l_1$ is the Weyl orbit of $\l_1$, $W(\l_1,\mu)=\{w\in W|w\l_1=\mu\}$, and the last expression represents the Laurent series expansion of the $w$-sum in $\e$.  Thus
the leading coefficient $\lim_{\e\rightarrow 0}\e^{-n}E(\l_1+\e\l_2,g)$ in its Laurent expansion in $\e$ is an automorphic form with constant term
\begin{equation}\label{grouping2}
   \int_{N(F)\backslash N(\A_F)} \lim_{\e\,\rightarrow\, 0}\e^{-n}\,E(\l_1+\e\l_2,ng)\,dn \ \ = \ \ \sum_{\mu\,\in\,W\l_1}a(g)^{\mu+\rho} \,C(\mu,n,g)\,.
\end{equation}
The interchange of the limit and integration here is justified because the limit may be computed as an integral over a compact space using Cauchy's theorem.
Note that certain $\mu$ in the Weyl orbit $W\l_1$ may have $C(\mu,n,g)\equiv 0$ as a function of $g$, while others may have a logarithmic factors; however, the polynomial growth in $g$ is always determined by the factor $a(g)^{\mu+\rho}$.
Langlands \cite[\S5]{Langlands} gave the condition that because $\lim_{\e\rightarrow 0}\e^{-n}E(\l_1+\e\l_2,g)$ is ``concentrated along $B$'', it is square-integrable if and only if the inequality
\begin{equation}\label{langlandscondition}
    \langle \mu,\omega_i\rangle \ < \ 0 \ \ \ \text{for any fundamental weight~}\omega_i\,
\end{equation}
holds for each   $\mu\in W\l_1$ in (\ref{grouping2}) such that $C(\mu,n,g)\not\equiv0$.

\section{Computational method}\label{sec:method}

Explicit computations with (\ref{constanttermformula}) are frequently unwieldy with large Weyl groups $W$.  However, a perturbation method was introduced in \cite[\S7]{GMV} which manageably  reduces the size involved. In terms of the parameterization $\l=\l_1+\e\l_2$ of the previous section, it corresponds to taking $\l_1=2s\omega_j-\rho$ and $\l_2=\omega_j$ for some fixed $s\in \C$ and fundamental weight $\omega_j$.  Such a $\l$ corresponds to a maximal parabolic Eisenstein series induced from the trivial representation.  Of course few half-markings $\l_0({\mathcal O}^\vee)$ satisfy these hypothesis for $\l_1$, but each turns out to have a Weyl translate  which does.  By the functional equation (\ref{eisfe}) the limiting value $\lim_{\e\rightarrow 0}\e^{-n}E(\l_1+\e\l_2,g)$ corresponds to a nonzero multiple of a residue at $\l_0({\mathcal O}^\vee)$.  Hence there is no less of generality in studying the series expansion near $\l=\l_1$ instead of $\l=\l_0({\cal O}^\vee)$.

With our choice of $\l=(2s+\e)\omega_j-\rho$ the inner products with simple coroots are
\begin{equation}\label{innerprod}
    \langle \l,\a_i^\vee\rangle \ = \, \ \left\{
                                          \begin{array}{ll}
                                            -1, & i\,\neq\,j \\
                                            2s+\e-1, & i\,=\,j\,.
                                          \end{array}
                                        \right.
\end{equation}
If $\e$ has sufficiently negative real part then
\begin{equation}\label{observation}
    c(\langle \l,\a^\vee\rangle) \ \ =  \ \ 0 \ \ \ \ \text{if} \ \ \ \langle \l,\a^\vee\rangle \,=\,-1\,,
\end{equation}
while no other factors $c(\langle \l,\a^\vee\rangle)$ in (\ref{Mwlambdadef}) have poles.  Thus by analytic continuation in $\e$ a term $M(w,\l)\equiv 0$ unless $w$ lies in the set
\begin{equation}\label{Wreldef}
    W_{\text{rel}} \ \ := \ \ \{ \,w\,\in\,W\,|\, w\a_i>0 \ \ \text{for all} \ i\,\neq\,j\,\}.
\end{equation}
As the tables in section~\ref{sec:data} show, $W_{\text{rel}}$ is significantly smaller than $W$:~its order equals $\#(W/W_M)$, where $W_M$ is the Weyl group of the root system spanned by $\{\a_i|i\neq j\}$.  Because it  is so much smaller than $W$,
it has fewer translates $W_{\text{rel}}\l_1$ of $\l_1$.  The tables also indicate an overwhelming majority of the  $W_{\text{rel}}$-translates    satisfy (\ref{langlandscondition}).  Square-integrability is therefore reduced to showing the existence of an integer $m$ such that
\begin{equation}\label{punchline}
   \aligned
\text{i)} \ \  & C(\mu,m,g)\,\not\equiv 0 \, \ \ \text{for some ~$\mu\,\in\,W_{\text{rel}}\l_1$~ satisfying (\ref{langlandscondition}), and }
\\
\text{ii)} \  \ & C(\mu,m',g) \equiv 0 \  \ \text{for all ~$m'\,\le\,m$~ and all ~$\mu\,\in\,W_{\text{rel}}\l_1$~ not satisfying (\ref{langlandscondition}).}
\endaligned
\end{equation}
Indeed,  there then exists some $n\le m$ such that $\lim_{\e\rightarrow 0}\e^{-n}E(\l_1+\e\l_2,g)$ lies in $L^2(G(F)\backslash G(\A_F))$, $n$ being the least integer such that $C(\mu,n,g)\not\equiv 0$ for some $\mu\in W_{\text{rel}}\l_1$ satisfying (\ref{langlandscondition}).

Thus the method boils down to the ability to efficiently compute certain coefficients $C(\mu,k,g)$ in the Laurent expansion (\ref{grouping1}).  We used exact, symbolic calculation with integer arithmetic (which is rigorous), and have made the programs available for download from the author's website at \url{http://www.math.rutgers.edu/~sdmiller/L2}.  We used Mathematica v.8, though many other software packages would have sufficed.  Mathematica is capable of exact calculations with the exact formula for  $c(s)$ when $F=\Q$, though they become cumbersome (and do not apply to general number fields);~it also crashes both Windows and Unix operating systems for large symbolic computations such as ours, forcing us to rewrite them differently using   some software tricks which we now describe.   We computed $c(  \langle \l,\a_i^\vee\rangle)=c(  \langle \l_1,\a_i^\vee\rangle+ \e\langle \l_2,\a_i^\vee\rangle)$ differently depending on the value of  $\langle \l_1,\a_i^\vee\rangle$ (which is always an integer in our examples):~near $s\neq \pm1$ we formally write $c(s)=-e^{c_\ell(s)-c_\ell(-s)}$ in order to satisfy the properties $c(s)c(-s)=1$ and $c(0)=-1$, while at near $s=\pm1$ we use the similar expression $c(s)=\mp(s\mp 1)^{\mp 1} e^{\mp c_{\ell,1}(1\mp s)}$ that takes into account the simple zero or pole.  These formal expressions are valid in a neighborhood of the relevant integer, and thus can be rigorously used to derive formulas for Laurent series expansions in $\e$.
As a result each $c(  \langle \l,\a_i^\vee\rangle)$ can be written as a power of $\epsilon$ times the exponential of a formal  function of $\langle \l_1,\a_i^\vee\rangle+ \e\langle \l_2,\a_i^\vee\rangle$.  By invariance properties we may take $g$ in (\ref{grouping1}) to lie in $A(F_\infty)$, and thus be the exponential of an element $H\in{\frak a}(F_\infty)$.  Hence $M(w,\l_1+\e\l_2)a(g)^{\e w \l_2}$ can always be written as a power of $\e$ times an exponential.
We used the observation that this power of $\e$ is constant over $W(\l_1,\mu)$ for a fixed $\mu$ in $W_{\text{rel}}\l_1$ to further speed up the computations.  These powers of $\e$ come from $w$ for which $\langle w\l_1,\a^\vee\rangle$ is equal to $-1$ or $1$ for some positive root $\a$ flipped by $w$.

At this point, for a fixed $\mu\in W_{\text{rel}}\l_1$ each $M(w,\l_1+\e\l_2)a(g)^{\e w \l_2}$, $w\in W(\l_1,\mu)$, is a fixed power of $\e$ times an overall multiplicative constant (which comes from the terms with $\langle \l_1,\a_i^\vee\rangle\in\{-1,0,1\}$) and the exponential of a function of $\e$.  To compute the Laurent series we compute the  Taylor  series development of that function in $\e$, and then use the power series $e^x=\sum_{k\ge 0}x^k/k!$ to derive the Taylor series of the full exponential term in $\e$.  We then multiply by the overall multiplicative constant and fixed power of $\e$, and then sum over $w\in W(\l_1,\mu)$.  This computes Laurent series of the inner sum in (\ref{grouping1}).  In each case, the first nonvanishing $C(\mu,k,g)$ has a nonzero polynomial dependence in $H$; this makes the nonvanishing independent of $c(\cdot)$ and hence also independent of the ground field $F$.

Ultimately, the calculation boils down to calculating averages of polynomials over highly symmetric, finite subsets of euclidean space -- that is,  a {\em design computation}.  Its difficulty stems from the appearance of large symmetric subsets which cannot be distinguished from an equidistributed set until a high degree polynomial is taken.  Such sets arise in these constant term calculations because of their similarity to the Weyl denominator formula.  In the largest cases we took advantage of repeated occurrences  of  certain values of $M(w,\l_1+\e\l_2)$ and its derivatives at $\e=0$, by compressing them into  new variables.  This dramatically sped up symbolic calculations and reduced {\sc RAM} requirements.

The computations were carried out on a Dell PowerEdge server, and required up to 45 GB of RAM. The lengthiest example (the first line in the $G=E_8$ table below) took about 3 days utilizing 8 CPUs, the vast majority of which was spent  verifying that Mathematica had correctly manipulated a symbolic expression.  Additionally, nearly all of these calculations have been reproduced using  completely independent code using LiE and Sage, with full agreement.  This code is also available from the author's website.

\section{Results of the calculation}\label{sec:data}

Below we present some details of the calculation  for Chevalley groups of type $E_6$, $E_7$,  $E_8$, and $F_4$.

\vspace{.2cm}

\noindent {\bf Legend for columns}:

$2\l_0(\cal O^\vee)$: the weighted marking of a distinguished (and hence even) coadjoint nilpotent orbit $\cal O^\vee$.  The integers $\langle 2\l_0({\cal O}^\vee),\a_1^\vee\rangle,\langle 2\l_0({\cal O}^\vee),\a_2^\vee\rangle,
\ldots$ are strung together in the standard Bourbaki numbering.

$\cal O$: the wavefront set of the spherical constituent of $V_{\l_0(\cal O^\vee)}$ (for $v$ archimedean).

$\l_1$: a Weyl-equivalent point to $\l_0(\cal O^\vee)$ which is more useful for our computational purposes in section~\ref{sec:method}, listed as $[\langle \l_1,\a_1^\vee\rangle,\langle \l_1,\a_2^\vee\rangle,
\ldots]$.

$W_{\text{rel}}$:  Weyl elements which give nonzero contributions to the constant term under the deformation $\l_1+\e\l_2$ (defined in (\ref{Wreldef})).

$\#\cap -{\cal C}^\vee$: The first number indicates how many terms in $W_{\text{rel}}\l_1$ lie in Langlands' region (\ref{langlandscondition}); the second number indicates how many do not  satisfy (\ref{langlandscondition}); and the third indicates how many of these would instead not satisfy (\ref{langlandscondition}) if the condition's inequality was weakened from $\langle \l,\omega_i\rangle < 0$ to $\langle \l,\omega_i\rangle \le 0$ for each $i$.

ord: the order of vanishing of $E(\l_1+\e\l_2,g)$ at $\e=0$.  This was called $n$ in section~\ref{sec:Langlands}.  The first entry for the $E_8$ table has an inequality; we did not determine $n$ here but found a value of  $m=3$ (in the notation of (\ref{punchline})).

\begin{center}
$E_6$

\begin{tabular}{cccccc}
\hline
 \multicolumn{1}{|c|}{$2\lambda_0(\mathcal O^\vee$)} &  \multicolumn{1}{c|}{${\mathcal O}$} &  \multicolumn{1}{c|}{$\l_1$} &  \multicolumn{1}{c|}{$\#W_{\text{rel}}$} &  \multicolumn{1}{c|}{$\#\cap-{\mathcal C}^{\vee}$} &  \multicolumn{1}{c|}{ord} \\
 \cline{1-1} \cline{2-2} \cline{3-3} \cline{4-4} \cline{5-5} \cline{6-6}
 \multicolumn{1}{|c|}{200202} &  \multicolumn{1}{c|}{$A_2$} &  \multicolumn{1}{l|}{[-1,4,-1,-1,-1,-1]} &  \multicolumn{1}{c|}{72} &  \multicolumn{1}{c|}{44/1/0} &  \multicolumn{1}{c|}{0}  \\
 \cline{1-1} \cline{2-2} \cline{3-3} \cline{4-4} \cline{5-5} \cline{6-6}
 \multicolumn{1}{|c|}{222022} &  \multicolumn{1}{c|}{$A_1$} &  \multicolumn{1}{l|}{[2,-1,-1,-1,-1,-1]} &  \multicolumn{1}{c|}{27} &  \multicolumn{1}{c|}{24/2/1} &  \multicolumn{1}{c|}{0}  \\
 \cline{1-1} \cline{2-2} \cline{3-3} \cline{4-4} \cline{5-5} \cline{6-6}
 \multicolumn{1}{|c|}{222222} &  \multicolumn{1}{c|}{$0$} &  \multicolumn{1}{l|}{[-1,-1,-1,-1,-1,-1]} &  \multicolumn{1}{c|}{1} &  \multicolumn{1}{c|}{1/0/0} &  \multicolumn{1}{c|}{0}   \\
 \cline{1-1} \cline{2-2} \cline{3-3} \cline{4-4} \cline{5-5} \cline{6-6}
\end{tabular}\end{center}

\begin{center}
$E_7$

\begin{tabular}{cccccc}
\hline
 \multicolumn{1}{|c|}{$2\lambda_0(\mathcal O^\vee$)} &  \multicolumn{1}{c|}{${\mathcal O}$} &  \multicolumn{1}{c|}{$\l_1$} &  \multicolumn{1}{c|}{$\#W_{\text{rel}}$} &  \multicolumn{1}{c|}{$\#\cap-{\mathcal C}^{\vee}$} &  \multicolumn{1}{c|}{ord} \\
 \cline{1-1} \cline{2-2} \cline{3-3} \cline{4-4} \cline{5-5} \cline{6-6}
 \multicolumn{1}{|c|}{0002002} &  \multicolumn{1}{c|}{$D_4(a_1)$} &  \multicolumn{1}{c|}{[-1,-1,4,-1,-1,-1,-1]} &  \multicolumn{1}{c|}{2016} &  \multicolumn{1}{c|}{638/27/2} &  \multicolumn{1}{c|}{1}  \\
 \cline{1-1} \cline{2-2} \cline{3-3} \cline{4-4} \cline{5-5} \cline{6-6}
 \multicolumn{1}{|c|}{2002002} &  \multicolumn{1}{c|}{$A_2+2A_1$} &  \multicolumn{1}{c|}{[-1,5,-1,-1,-1,-1,-1]} &  \multicolumn{1}{c|}{576} &  \multicolumn{1}{c|}{292/2/1} &  \multicolumn{1}{c|}{0}  \\
 \cline{1-1} \cline{2-2} \cline{3-3} \cline{4-4} \cline{5-5} \cline{6-6}
 \multicolumn{1}{|c|}{2002022} &  \multicolumn{1}{c|}{$A_2$} &  \multicolumn{1}{c|}{[7,-1,-1,-1,-1,-1,-1]} &  \multicolumn{1}{c|}{126} &  \multicolumn{1}{c|}{90/1/0} &  \multicolumn{1}{c|}{0} \\
 \cline{1-1} \cline{2-2} \cline{3-3} \cline{4-4} \cline{5-5} \cline{6-6}
 \multicolumn{1}{|c|}{2220202} &  \multicolumn{1}{c|}{$2A_1$} &  \multicolumn{1}{c|}{[4,-1,-1,-1,-1,-1,-1]} &  \multicolumn{1}{c|}{126} &  \multicolumn{1}{c|}{115/3/1} &  \multicolumn{1}{c|}{0} \\
 \cline{1-1} \cline{2-2} \cline{3-3} \cline{4-4} \cline{5-5} \cline{6-6}
 \multicolumn{1}{|c|}{2220222} &  \multicolumn{1}{c|}{$A_1$} &  \multicolumn{1}{c|}{[2,-1,-1,-1,-1,-1,-1]} &  \multicolumn{1}{c|}{126} &  \multicolumn{1}{c|}{97/28/0} &  \multicolumn{1}{c|}{0} \\
 \cline{1-1} \cline{2-2} \cline{3-3} \cline{4-4} \cline{5-5} \cline{6-6}
 \multicolumn{1}{|c|}{2222222} &  \multicolumn{1}{c|}{$0$} &  \multicolumn{1}{c|}{[-1,-1,-1,-1,-1,-1,-1]} &  \multicolumn{1}{c|}{1} &  \multicolumn{1}{c|}{1/0/0} &  \multicolumn{1}{c|}{0} \\
 \cline{1-1} \cline{2-2} \cline{3-3} \cline{4-4} \cline{5-5} \cline{6-6}
\end{tabular}\end{center}

\begin{center}
$E_8$

\begin{tabular}{cccccc}
\hline
 \multicolumn{1}{|c|}{$2\lambda_0(\mathcal O^\vee$)} &  \multicolumn{1}{c|}{${\mathcal O}$} &  \multicolumn{1}{c|}{$\l_1$} &  \multicolumn{1}{c|}{$\#W_{\text{rel}}$} &  \multicolumn{1}{c|}{$\#\cap-{\mathcal C}^{\vee}$} &  \multicolumn{1}{c|}{ord} \\
\cline{1-1} \cline{2-2} \cline{3-3} \cline{4-4} \cline{5-5} \cline{6-6}
 \multicolumn{1}{|c|}{00002000} &  \multicolumn{1}{c|}{$E_8(a_7)$} &  \multicolumn{1}{c|}{[-1,-1,-1,-1,4,-1,-1,-1]} &  \multicolumn{1}{c|}{241920} &  \multicolumn{1}{c|}{18881/3897/1329} &  \multicolumn{1}{c|}{$\le 3$} \\
 \cline{1-1} \cline{2-2} \cline{3-3} \cline{4-4} \cline{5-5} \cline{6-6}
 \multicolumn{1}{|c|}{00020002} &  \multicolumn{1}{c|}{$D_4(a_1)+A_2$} &  \multicolumn{1}{c|}{[-1,7,-1,-1,-1,-1,-1,-1]} &  \multicolumn{1}{c|}{17280} &  \multicolumn{1}{c|}{3638/2/1} &  \multicolumn{1}{c|}{0}  \\
 \cline{1-1} \cline{2-2} \cline{3-3} \cline{4-4} \cline{5-5} \cline{6-6}
 \multicolumn{1}{|c|}{00020020} &  \multicolumn{1}{c|}{$D_4(a_1)+A_1$} &  \multicolumn{1}{c|}{[-1,6,-1,-1,-1,-1,-1,-1]} &  \multicolumn{1}{c|}{17280} &  \multicolumn{1}{c|}{8902/603/22} &  \multicolumn{1}{c|}{1} \\
 \cline{1-1} \cline{2-2} \cline{3-3} \cline{4-4} \cline{5-5} \cline{6-6}
 \multicolumn{1}{|c|}{00020022} &  \multicolumn{1}{c|}{$D_4(a_1)$} &  \multicolumn{1}{c|}{[-1,-1,-1,-1,-1,-1,8,-1]} &  \multicolumn{1}{c|}{6720} &  \multicolumn{1}{c|}{3143/49/1} &  \multicolumn{1}{c|}{1} \\
 \cline{1-1} \cline{2-2} \cline{3-3} \cline{4-4} \cline{5-5} \cline{6-6}
 \multicolumn{1}{|c|}{20020020} &  \multicolumn{1}{c|}{$2A_2$} &  \multicolumn{1}{c|}{[10,-1,-1,-1,-1,-1,-1,-1]} &  \multicolumn{1}{c|}{2160} &  \multicolumn{1}{c|}{1099/1/0} &  \multicolumn{1}{c|}{0} \\
 \cline{1-1} \cline{2-2} \cline{3-3} \cline{4-4} \cline{5-5} \cline{6-6}
 \multicolumn{1}{|c|}{20020022} &  \multicolumn{1}{c|}{$A_2+2A_1$} &  \multicolumn{1}{c|}{[8,-1,-1,-1,-1,-1,-1,-1]} &  \multicolumn{1}{c|}{2160} &  \multicolumn{1}{c|}{1647/13/4} &  \multicolumn{1}{c|}{0}\\
 \cline{1-1} \cline{2-2} \cline{3-3} \cline{4-4} \cline{5-5} \cline{6-6}
 \multicolumn{1}{|c|}{20020202} &  \multicolumn{1}{c|}{$A_2+A_1$} &  \multicolumn{1}{c|}{[7,-1,-1,-1,-1,-1,-1,-1]} &  \multicolumn{1}{c|}{2160} &  \multicolumn{1}{c|}{1763/157/26} &  \multicolumn{1}{c|}{0}  \\
 \cline{1-1} \cline{2-2} \cline{3-3} \cline{4-4} \cline{5-5} \cline{6-6}
 \multicolumn{1}{|c|}{20020222} &  \multicolumn{1}{c|}{$A_2$} &  \multicolumn{1}{c|}{[-1,-1,-1,-1,-1,-1,-1,13]} &  \multicolumn{1}{c|}{240} &  \multicolumn{1}{c|}{195/1/0} &  \multicolumn{1}{c|}{0} \\
 \cline{1-1} \cline{2-2} \cline{3-3} \cline{4-4} \cline{5-5} \cline{6-6}
 \multicolumn{1}{|c|}{22202022} &  \multicolumn{1}{c|}{$2A_1$} &  \multicolumn{1}{c|}{[-1,-1,-1,-1,-1,-1,-1,8]} &  \multicolumn{1}{c|}{240} &  \multicolumn{1}{c|}{229/2/0} &  \multicolumn{1}{c|}{0}  \\
 \cline{1-1} \cline{2-2} \cline{3-3} \cline{4-4} \cline{5-5} \cline{6-6}
 \multicolumn{1}{|c|}{22202222} &  \multicolumn{1}{c|}{$A_1$} &  \multicolumn{1}{c|}{[-1,-1,-1,-1,-1,-1,-1,4]} &  \multicolumn{1}{c|}{240} &  \multicolumn{1}{c|}{224/15/0} &  \multicolumn{1}{c|}{0} \\
 \cline{1-1} \cline{2-2} \cline{3-3} \cline{4-4} \cline{5-5} \cline{6-6}
 \multicolumn{1}{|c|}{22222222} &  \multicolumn{1}{c|}{$0$} &  \multicolumn{1}{l|}{[-1,-1,-1,-1,-1,-1,-1,-1]} &  \multicolumn{1}{c|}{1} &  \multicolumn{1}{c|}{1/0/0} &  \multicolumn{1}{c|}{0} \\
 \cline{1-1} \cline{2-2} \cline{3-3} \cline{4-4} \cline{5-5} \cline{6-6}
\end{tabular}\end{center}

\begin{center}
$F_4$

\begin{tabular}{cccccc}
\hline
 \multicolumn{1}{|c|}{$2\lambda_0(\mathcal O^\vee$)} &  \multicolumn{1}{c|}{${\mathcal O}$} &  \multicolumn{1}{c|}{$\l_1$} &  \multicolumn{1}{c|}{$\#W_{\text{rel}}$} &  \multicolumn{1}{c|}{$\#\cap-{\mathcal C}^{\vee}$} &  \multicolumn{1}{c|}{ord} \\
\cline{1-1} \cline{2-2} \cline{3-3} \cline{4-4} \cline{5-5} \cline{6-6}
 \multicolumn{1}{|c|}{0020} &  \multicolumn{1}{c|}{$F_4(a_3)$} &  \multicolumn{1}{c|}{[-1,1,-1,-1]} &  \multicolumn{1}{c|}{96} &  \multicolumn{1}{c|}{23/24/9} &  \multicolumn{1}{c|}{2} \\
 \cline{1-1} \cline{2-2} \cline{3-3} \cline{4-4} \cline{5-5} \cline{6-6}
 \multicolumn{1}{|c|}{2020} &  \multicolumn{1}{c|}{$A_1+A_1s$} &  \multicolumn{1}{c|}{[2,-1,-1,-1]} &  \multicolumn{1}{c|}{24} &  \multicolumn{1}{c|}{15/2/1} &  \multicolumn{1}{c|}{0} \\
 \cline{1-1} \cline{2-2} \cline{3-3} \cline{4-4} \cline{5-5} \cline{6-6}
 \multicolumn{1}{|c|}{2022} &  \multicolumn{1}{c|}{$A_1s$} &  \multicolumn{1}{c|}{[1,-1,-1,-1]} &  \multicolumn{1}{c|}{24} &  \multicolumn{1}{c|}{17/6/0} &  \multicolumn{1}{c|}{0} \\
 \cline{1-1} \cline{2-2} \cline{3-3} \cline{4-4} \cline{5-5} \cline{6-6}
 \multicolumn{1}{|c|}{2222} &  \multicolumn{1}{c|}{$0$} &  \multicolumn{1}{c|}{[-1,-1,-1,-1]} &  \multicolumn{1}{c|}{1} &  \multicolumn{1}{c|}{1/0/0} &  \multicolumn{1}{c|}{0}  \\
 \cline{1-1} \cline{2-2} \cline{3-3} \cline{4-4} \cline{5-5} \cline{6-6}
\end{tabular}
\end{center}

\end{document}